\documentclass[11 pt, reqno]{amsart}

\usepackage{amsmath}
\usepackage{amsthm}
\usepackage{amssymb}
\usepackage{amsfonts}
\usepackage{mathtools}
\usepackage{mathrsfs}
\usepackage{graphicx}
\usepackage[hidelinks, allcolors=blue]{hyperref}

\usepackage[justification=centering]{caption}

\newtheorem{theorem}{Theorem}[section]
\newtheorem{proposition}[theorem]{Proposition}
\newtheorem{question}[theorem]{Question}
\newtheorem{lemma}[theorem]{Lemma}
\newtheorem{corollary}[theorem]{Corollary}
\theoremstyle{definition}
\newtheorem{definition}[theorem]{Definition}

\renewcommand{\phi}{\varphi}
\renewcommand{\rho}{\varrho}
\renewcommand{\epsilon}{\varepsilon}

\newcommand{\susbet}{\subset} 

\newcommand{\abs}[1]{\lvert#1\rvert}
\newcommand{\diam}{\operatorname{diam}}

\DeclareMathOperator{\CAT}{CAT}

\DeclareMathOperator{\R}{\mathbb{R}}

\DeclareMathOperator{\Z}{\mathbb{Z}}

\DeclareMathOperator{\conv}{\sigma--conv}
\DeclareMathOperator{\convt}{\widetilde{\sigma}--conv}
\DeclareMathOperator{\convm}{m--conv}
\DeclareMathOperator{\convNeu}{conv}

\DeclareMathOperator{\closure}{closure}

\newcommand{\olsi}[1]{\,\overline{\!{#1}}}

\numberwithin{equation}{section}
\counterwithin{figure}{section}
\title[A non-compact convex hull]{A non-compact convex hull \\ in generalized non-positive curvature}
\author{Giuliano Basso,  Yannick Krifka, and Elefterios Soultanis}


\begin{document}

\begin{abstract}
In this article, we are interested in metric spaces that satisfy a weak non-positive curvature condition in the sense that they admit a conical geodesic bicombing. We show that the analog of a question of Gromov about compactness properties of convex hulls in \(\CAT(0)\) spaces has a negative answer in this setting. Specifically, we prove that there exists a complete metric space \(X\) that admits a conical bicombing \(\sigma\) such that \(X\) has a finite subset whose closed \(\sigma\)-convex hull is not compact.
\end{abstract}

\maketitle

 \vspace{-0.5em}

\section{Introduction}
Let \((X, d)\) be a metric space and \(\sigma_{xy}\colon [0,1]\to X\) be a geodesic in \(X\) such that \(\sigma_{xy}(0)=x\) and \(\sigma_{xy}(1)=y\). A selection \(\sigma\) of such a geodesic \(\sigma_{xy}\) for all \(x\), \(y\in X\) is called a \textit{(geodesic) bicombing}. The idea to distinguish certain geodesics of a metric space and to investigate geometric questions in such a setting goes back to Busemann and Phadke \cite{busemann--1987} (see also \cite{ lytchak-2007, kleiner--1999,  lang-2013, okon--2000}). The term bicombing was coined by Thurston \cite[p. 84]{word-processing-1992}, and originally bicombings consisting of quasi-geodesics were mainly studied in the context of geometric group theory \cite{alonso-bridson-1995, gromov-1993, huck-rosebnrock-1993}. Following Descombes and Lang \cite{descombes-lang-2015}, we say that a bicombing \(\sigma\) is \textit{conical} if for all \(x\), \(y\), \(x'\), \(y'\in X\),
\begin{equation}\label{eq:conical}
d(\sigma_{xy}(t), \sigma_{x'y'}(t)) \leq (1-t)\cdot d(x, x')+t\cdot d(y, y')
\end{equation}
for all \(t\in [0,1]\). Hence, the distance between the points on the \(\sigma\)--geodesics at time \(t\) is smaller than the linear interpolation of the distances between the endpoints. It is well-known that non-positively curved metric spaces such as \(\CAT(0)\) spaces or, more generally, Busemann spaces admit a conical bicombing. See  \cite{alexandrov-2019, ballmann--1995, bridson-haefliger-1999, burago-2001, papadopoulos--2014} for excellent introductions to these spaces, which are synthetic generalizations of Riemannian manifolds of non-positive sectional curvature. Recently, metric spaces that admit a conical bicombing have gained some interest and found a diverse range of applications. They can be used to study Helly groups \cite{chalopin2020helly, huang-osajda-2021}, Lipschitz extension problems \cite{creutz-2020, mendel-naor-2013, naor2021extension}, and metric fixed point problems \cite{karlsson2024metric, ito--1979, kohlenabch--2010}. 

In this article, we use conical bicombings as a test case to study compactness properties of convex hulls in non-positively curved metric spaces. The starting point of our considerations is the following intriguing question due to Gromov \cite[6.B\(_1\)(f)]{gromov-1993}.

\begin{question}[Gromov]\label{qe:gromov}
Let \(X\) be a complete \(\CAT(0)\) space and \(K\subset X\) a compact subset. Is it true that the closed convex hull of \(K\) is compact?
\end{question} 

This question has been popularized by Petrunin (see \cite{petrunin-MO-6627} and \cite[p.\ 77]{petrunin2009pigtikal}). A metric space is called \textit{proper} if every of its closed bounded subsets is compact. For proper metric spaces, Question~\ref{qe:gromov} has an affirmative answer. However, when \(X\) is not proper, it seems very difficult to answer. We remark that using standard techniques from \(\CAT(0)\) geometry one can show that Question~\ref{qe:gromov} has an affirmative answer if and only if it has an affirmative answer for finite subsets (see \cite{duchesne-2018}). However, already the case of three-point subsets is completely open. We refer to \cite{lytchak-petrunin-2022} for more information.

Clearly, Question~\ref{qe:gromov} can also be stated for spaces with a conical bicombing.
 We say that \(C\subset X\) is \textit{\(\sigma\)-convex} if  for all \(x\), \(y\in C\), the geodesic \(\sigma_{xy}\) is contained in \(C\). For any \(A\subset X\), we consider the closed \(\sigma\)-convex hull of \(A\), 
\[
\conv(A)= \bigcap C,
\]
where the intersection is taken over all closed \(\sigma\)-convex subsets \(C\susbet X\) containing \(A\). If \(X\) is a complete metric space and \(\sigma\) a conical bicombing on \(X\), then the pair \((X,\sigma)\) is called a \textit{space of generalized non-positive curvature}. Our main result shows that for such spaces the analog of Question~\ref{qe:gromov} has a negative answer.

\begin{theorem}\label{thm:non-cpmpactness}
There exists a space of generalized non-positive curvature that has a finite subset whose closed \(\sigma\)-convex hull is not compact.
\end{theorem}

Thus, to obtain an affirmative answer to Gromov's question, more than just the convexity properties of the metric must be used.

We remark that it follows from results of \cite{basso2020extending} that there is a metric space as in Theorem~\ref{thm:non-cpmpactness} which is additionally an injective metric space; see Theorem~\ref{thm:extension-to-injective-metric-spaces} below. Injective spaces are prime examples of metric spaces with a conical bicombing. Descombes and Lang \cite{descombes-lang-2015} showed that injective metric spaces of finite combinatorial dimension admit a bicombing which satisfies a stronger convexity property than \eqref{eq:conical}. We do not know whether Theorem~\ref{thm:non-cpmpactness} also holds  for these bicombings. 

The main work behind the construction of the metric space \(X\) in Theorem~\ref{thm:non-cpmpactness} is done in a discrete setting. In Section~\ref{sec:appending-midpoints},  we start with a finite set \(V_0\) and iteratively append midpoints to it. This yields a nested sequence 
\[
V_0\subset V_1 \subset \dotsm \subset V_n \subset \dotsm 
\]
of finite sets such that their union \(V\) is closed under a midpoint map 
\[
m\colon V \times V \to V.
\]
We then metrize \(V\) in such a way that \(m\) satisfies a discrete analog of the conical inequality \eqref{eq:conical}.  In Section~\ref{sec:not-totally-bounded}, we explicitly construct an infinite \(\epsilon\)-separated subset of \((V, d)\) for \(\epsilon=\tfrac{1}{4}\). Hence, \(V\) is not totally bounded and thus its metric completion \(X\) is not compact. Finally, in Section~\ref{sec:conical-midpoint}, we use a classical construction that goes back to Menger \cite{menger-1928} to show that \(m\) induces a conical bicombing \(\sigma\) on \(X\) in a natural way. Moreover, by construction, it follows that \(X\) is the closed \(\sigma\)-convex hull of \(V_0\). Hence, \((X, \sigma)\) is a space of generalized non-positive curvature with the properties we are looking for, proving Theorem~\ref{thm:non-cpmpactness}.

The original idea behind our construction was to ensure the existence of the initial object \(X\) in the following theorem. 

\begin{theorem}\label{thm:universal}

For each positive integer \(n\) there exists a complete metric space \(X\) with the following properties: 

\begin{enumerate}
    \item \(X\) admits a conical bicombing \(\sigma\) and there is an \(n\)-point subset \(A_0\subset X\) such that \(X\) is the closed \(\sigma\)-convex hull of \(A_0\);
    \item whenever \(A\subset Y\) is an \(n\)-point subset of a complete \(\CAT(0)\) space \(Y\), then there exists a Lipschitz map \(\Phi\colon X \to Y\) such that \(\Phi(X)\) is convex and contains \(A\). 
\end{enumerate} 
\end{theorem}

We actually prove a stronger statement than Theorem~\ref{thm:universal}. Instead of complete \(\CAT(0)\) spaces \(Y\), more general non-positively curved target spaces can be considered. See Theorem~\ref{thm:universal-with-general-targets} below for the exact statement. We remark that, by construction, \(\convNeu(A)\susbet \closure( \Phi(X))\). Therefore, if \(\Phi(X)\) is precompact, then the closed convex hull of \(A\) is compact. Given this relation, it seems reasonable to suspect that \(X\) is non-compact. As it turns out, this is the case for every \(n>1\); see Theorem~\ref{thm:Z-non-cpt}.

One may of course wonder whether there also exists such a space \(X\) as above, which is in addition a complete \(\CAT(0)\) space. The existence of such spaces would reduce Gromov's question to the problem of deciding whether these spaces \(X\) are compact or not. If they are all compact, then Question~\ref{qe:gromov} has an affirmative answer. On the other hand, the non-compactness of \(X\) for some \(n\) gives a negative answer. However, our proof does not seem to be directly amenable for generating \(\CAT(0)\) spaces.

\subsection{Acknowledgements} The first named author is indebted to Urs Lang, Alexander Lytchak, and Stephan Stadler for useful discussions about convex hulls. We also thank the anonymous referees for their helpful suggestions, which prompted us to considerably simplify our original construction.

\section{Bicombings}

Let \((X, d)\) be a metric space. We say that \(\sigma\colon [0,1]\to X\) is a \textit{geodesic} if \(d(\sigma(s), \sigma(t))=\abs{s-t}\cdot d(\sigma(0), \sigma(1)) \) for all \(s\), \(t\in [0,1]\). A map 
\[
\sigma\colon X\times X \times [0,1]\to X
\]
is called a \textit{(geodesic) bicombing} if for all \(x\), \(y\in X\), the path \(\sigma_{xy}(\cdot)\colon [0,1]\to X\) defined by \(\sigma_{xy}(t)=\sigma(x, y, t)\) is a geodesic connecting \(x\) to \(y\).

We remark that, in contrast, a map \(\sigma\colon X\times [0,1]\to X\) is called a \textit{combing} with basepoint \(p\in X\) if for all \(x\in X\), the path \(\sigma(x, \cdot)\) is a geodesic connecting \(p\) to \(x\). However, we will not make use of this definition. Bicombings are also called \textit{system of good geodesics}; see \cite{ezawa2021visual, fukaya-2020, osajda-2009}. 

We say that \(\sigma\) is \textit{reversible} if \(\sigma_{xy}(t)=\sigma_{yx}(1-t)\) for all \(x\), \(y\in X\) and all \(t\in [0,1]\). In \cite[Proposition 1.3]{basso-miesch-2019} it is shown that any complete metric space with a conical bicombings also admits a conical reversible bicombing (see also \cite{descombes-2016} for an earlier result).

\section{Construction of \((V,d)\)}\label{sec:appending-midpoints}
Throughout this section we fix a positive integer \(j\). This \(j\) will correspond to the parameter from Theorem~\ref{thm:universal}.  Let \(d_0\) denote the discrete metric on \(V_0=\{0, \ldots, j-1\}\). For arbitrary sets \(a\), \(b\) we consider the set
\[
m(a, b)=
\begin{cases}
\{a, b\} & \text{ if } a\neq b, \\
a & \text{otherwise}.
\end{cases}
\]
For \(n>0\) we define recursively 
\[
V_n=\{m(a,b) : a, b\in V_{n-1}\}.
\]
For example, if \(j=2\), then we have 
\[
V_1=\{ 0, 1, \{0, 1\}\},
\]
and not \(V_1=\{\{0\}, \{1\}, \{0,1\}\}\). We have constructed an increasing sequence 
\[
V_0\subset V_1 \subset \dotsm \subset V_n \subset \dotsm .
\]
Let \(V\) denote the union of these sets. For any \(x\in V\setminus V_0\) there exist unique distinct \(a\), \(b\in V\) such that \(x=m(a,b)\), we call these elements the \textit{parents} of \(v\). 
\begin{definition}\label{def:def-of-d}
We define the function \(d\colon V\times V\to \R\) via the following recursive rules:
\begin{enumerate}
    \item if \(x\), \(y\in V_0\), then we set \(d(x, y)=d_0(x, y)\);\vspace{0.5em}
    \item if \(x\in V_0\) and \(y\in V\setminus V_0\) with parents \(c\), \(d\), then we set 
\[
d(x, y)=d(y, x)=\frac{1}{2}\big( d(x,c)+d(x, d) \big);
\]
\item and finally for arbitrary \(x\), \(y\in V\setminus V_0\) with parents \(a\), \(b\) and \(c\), \(d\), respectively, we set
\[
d(x, y)=\frac{1}{2}\cdot \min\{ A,\, A',\, B,\, B'\},
\]
where
\begin{align*}
A&=d(x,c)+d(x,d), & B&=d(a,c)+d(b,d), \\
A'&=d(y,a)+d(y,b), & B'&=d(a,d)+d(b,c).   
\end{align*}
\end{enumerate}
\end{definition}

At first glance it may not be clear that \(d\) is a well-defined function.  For \(x\in V\) we call the minimal \(n\geq 0\) such that \(x\in V_n\) the \textit{index} of \(x\). Using this notion it is not too difficult to show that \(d\) is well-defined. 

\begin{lemma}
The function \(d\colon V \times V \to \R\) from Definition~\ref{def:def-of-d} is well-defined. 
\end{lemma} 

\begin{proof}
Given \((x,y)\in V\times V\), we set \(\phi(x,y)=i(x)+i(y)\), where \(i(x)\) and \(i(y)\) denote the indices of \(x\) and \(y\), respectively. Clearly, \(d\) is well-defined on \(\phi^{-1}(0)\subset V\times V\). Now, if \(n=\phi(x,y)>0\), then the definition of \(d(x, y)\) only involves evaluations of \(d\) at tuples \((x', y')\) with  \(\phi(x', y')< n\). Hence, it follows by induction that \(d\) is a well-defined function of \(V\times V\).
\end{proof}

In the following, we show that \(d\) defines a metric on \(V\). 

\begin{proposition}\label{prop:v-is-metric}
\((V, d)\) is a metric space.
\end{proposition}

\begin{proof}
We use the shorthand notation \(xy:=d(x,y)\). Suppose \(d\) does not satisfy the triangle inequality. Then there exist \(x\), \(y\), \(z\in V\) such that
\begin{equation}\label{eq:disruption}
xz > zy+yx.
\end{equation}
Let \(x\), \(y\), \(z\in V\) be such that \eqref{eq:disruption} holds and the sum \(i(x)+i(y)+i(z)\) of the indices of \(x\), \(y\), and \(z\) is minimal. Without loss of generality, we may suppose that \(i(x)\leq i(z)\). To begin, we show that \(x\), \(y\), \(z\in V\setminus V_0\).

Suppose that \(z\in V_0\). Then it follows that \(x\in V_0\) and since \(d\) defines a metric on \(V_0\), we have that \(y\notin V_0\). Let \(c\), \(d\in V\) denote the parents of \(y\). We have
\[
zy+yx=\tfrac{1}{2}zc+\tfrac{1}{2}zd+\tfrac{1}{2}cx+\tfrac{1}{2}dx=\tfrac{1}{2}(zc+cx)+\tfrac{1}{2}(zd+dx),
\]
since \(i(c), i(d)<i(y)\), it thus follows that
\[
zy+yx\geq \tfrac{1}{2}zx+\tfrac{1}{2}zx=zx,
\]
which contradicts \eqref{eq:disruption}. Hence, it follows that \(i(z)>0\). Let \(e\), \(f\in V\) denote the parents of \(z\).

Now, suppose that \(i(y)=0\). Then
\begin{align*}
zy+yx&=\tfrac{1}{2} e y+\tfrac{1}{2}f y+\tfrac{1}{2}yx+ \tfrac{1}{2}yx\\
&=\tfrac{1}{2}(ey+yx)+\tfrac{1}{2}(fy+yx),
\end{align*}
and thus since \(i(e)\), \(i(f)<i(z)\), it follows that 
\[
zy+yx \geq \tfrac{1}{2}e x+\tfrac{1}{2} f x \geq zx,
\]
which contradicts \eqref{eq:disruption}. Hence, the case \(y\in V_0\) cannot occur. Let \(c\), \(d\in V\) denote the parents of \(y\).  

Finally, suppose that \(i(x)=0\). If
\[
zy=\tfrac{1}{2} z c+\tfrac{1}{2} z d,
\]
then using that \(x\in V_0\), we find that 
\[
zy+yx=\tfrac{1}{2}(zc+cx)+\tfrac{1}{2}( zd+dx)\geq z x,
\]
since \(i(c)\), \(i(d)<i(y)\). This contradicts \eqref{eq:disruption}. Thus,
\[
zy=\tfrac{1}{2}\min\{ ey+ fy,\, ec+ fd,\, ed+fc\},
\]
but then using that \(i(e)\), \(i(f)< i(z)\) and \(i(c)\), \(i(d)<i(y)\), as well as \(x\in V_0\), we find that
\[
zy+yx \geq \tfrac{1}{2}( ex+fx) =x z,
\]
which again contradicts \eqref{eq:disruption}. Therefore, the case \(x\in V_0\) cannot occur. 

To summarize, we have shown that \(x\), \(y\), \(z\in V\setminus V_0\). Let \(a\), \(b\in V\) denote the parents of \(x\). Recall that \(e\),\(f\) and \(c\), \(d\) denote the parents of \(z\) and \(y\), respectively. 

First, suppose that 
\begin{equation}\label{eq:first-possibility}
zy=\tfrac{1}{2} ey+\tfrac{1}{2}fy \quad \text{ or } \quad yx=\tfrac{1}{2} ay+\tfrac{1}{2}by.
\end{equation}
Hence, if \(zy=\tfrac{1}{2} ey+\tfrac{1}{2}fy\), then we have
\[
zy+yx=\tfrac{1}{2}(ey+yx)+\tfrac{1}{2}(fy+yx).
\]
But \(i(e)\), \(i(f)<i(z)\) and thus
\[
zy+yx \geq \tfrac{1}{2}ex+\tfrac{1}{2}fx \geq zx,
\]
which contradicts \eqref{eq:disruption}. Hence, the case \(zy=\tfrac{1}{2} ey+\tfrac{1}{2}fy\) cannot occur. By symmetry, the case \(yx=\tfrac{1}{2} ay+\tfrac{1}{2}by\) cannot occur as well. Hence, \eqref{eq:first-possibility} does not hold. Therefore, there exist \(u\), \(v\), \(p\), \(q\in V\) (which are not necessarily distinct) such that
\begin{equation}\label{eq:second-possibility}
zy=\tfrac{1}{2} uc+\tfrac{1}{2} vd \quad \text{and} \quad yx=\tfrac{1}{2} cp+\tfrac{1}{2} dq.
\end{equation}
Notice that \(\{u, v\}=\{z\}\) or \(\{u, v\}=\{e, f\}\), and analogously, \(\{p, q\}=\{x\}\) or \(\{p, q\}=\{a, b\}\). Consequently,
\[
zy+yx=\tfrac{1}{2}( uc+cp)+\tfrac{1}{2}(vd+dq).
\]
Now, since \(i(c)\), \(i(d) < i(y)\), we obtain
\[
zy+yx\geq \tfrac{1}{2} up+\tfrac{1}{2} vq.
\]
This implies that
\[
\tfrac{1}{2} up+\tfrac{1}{2} vq \geq zx,
\]
and so \(zy+yx\geq zx\), which contradicts \eqref{eq:disruption}. As a result, \eqref{eq:second-possibility} does not hold. But either \eqref{eq:first-possibility} or \eqref{eq:second-possibility} must necessarily be true. This is a contradiction and thus we obtain that points \(x\), \(y\), \(z\in V\) such that \eqref{eq:disruption} holds do not exit. This shows that \(d\) satisfies the triangle inequality on \(V\), and thus as it is symmetric and positive definite it defines a metric on \(V\), as desired. 
\end{proof}

\section{\((V, d)\) is not totally bounded}\label{sec:not-totally-bounded}
A metric space \(X\) is said to be \textit{totally bounded} if for every \(\epsilon>0\) there exists a finite subset \(F\susbet X\) such that 
\[
X=\bigcup_{x\in F} B(x, \epsilon).
\]
We recall that \(X\) is totally bounded if and only if its metric completion \(\olsi{X}\) is compact. In this section, we show by means of an explicit example that the metric space \(X=(V,d)\) constructed in the previous section is not totally bounded.

Recall that \(V_0=\{0, \ldots, j-1\}\), where \(j\) is a fixed positive integer. In the following, we suppose that \(j=2\). Clearly, if we show that \((V, d)\) is not totally bounded for \(j=2\), then \((V, d)\) is also not totally bounded if \(j>2\).

We abbreviate \(l=0\) and \(r=1\). Here, \(l\) stand for 'left' and \(r\) for 'right'. In the following, we construct a sequence \(y(n)\) of points in \(V\) that eventually moves far away from every point in \(V\). The definition of \(y(n)\) is recursive and involves two auxiliary 'triangular' sequences \(l(n, k)\) and \(r(n, k)\). We set
\[
l(1,1)=l \quad\text{ and }\quad r(1,1)=r
\]
as well as \(y(1)=m(l, r)\). Next, we define recursively, for \(n>1\) and \(k\in \{1, \ldots, n\}\),
\[
l(n, k)=
\begin{cases}
l& \text{if } k=1 \\
m\big(l(n-1, k), \,l(n, k-1)\big) & \text{if } 1 < k < n \\
m\big(y(n-1), l(n, n-1)\big) & \text{if } k=n
\end{cases}
\]
and analogously,
\[
r(n, k)=
\begin{cases}
r& \text{if } k=1 \\
m\big(r(n-1, k), \,r(n, k-1)\big) & \text{if } 1 < k < n \\
m\big(y(n-1), r(n, n-1)\big) & \text{if } k=n,
\end{cases}
\]
as well as
\[
y(n)=m(l(n,n), r(n, n)).
\]
See Figure~\ref{fig:drawing} for an illustration of the construction.
\begin{figure}[t]
    \centering
    \includegraphics[scale=1.5]{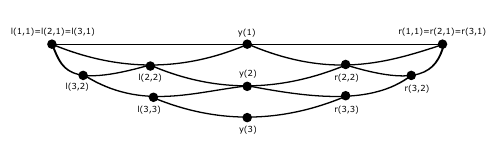}
    \caption{Schematic construction of the points \(y(1)\), \(y(2)\), and \(y(3)\).}
    \label{fig:drawing}
\end{figure} 
For example,
\begin{align*}
y(2)&=m\big(m(l,\,y(1)),\, m(r,\, y(1))\big) \\
y(3)&= m\big(m\big(m(l, \,l(2,2)), \,y(2)\big),\, m\big(y(2), m(r(2,2), r)\big)\big)
\end{align*}
It can be shown that the index of \(y(n)\) is equal to \(2n-1\). However, we will not need this and will therefore not prove it rigorously. It can be verified with a mathematical software of your choice that for \(n=3\), 
\begin{equation}\label{eq:some-conjecture}
d(y(n), y(1))=\max_{y\in M_{n}} \, d(y, y(1)), 
\end{equation}
where \(M_{n}\) is the set consisting of all points in \(V_{2n-1}\) that are contained in the midpoint set of \(l\) and \(r\), that is 
\[
M_{n}=\{y\in V_{2n-1} : \,  d(y, l)=d(y, r)=\tfrac{1}{2}\}.
\]
Presumably, \eqref{eq:some-conjecture} holds for all positive integers \(n\), and this is how we originally got interested in the sequence \(y(n)\). Notice that 
\[
d(y, y(1))\leq \tfrac{1}{2}d(y, l)+\tfrac{1}{2}d(y, r)=\tfrac{1}{2}
\]
for every \(y\in M_n\). Thus, by setting \(x=y(1)\) in the following proposition, we find that \(y(n)\) moves away from \(y(1)\) as far as possible.

\begin{proposition}\label{prop:main}
For every \(x\in V\),
\begin{equation}\label{eq:desired-limit}
\limsup_{n\to \infty} \,d(x, y(n)) \geq  \frac{1}{2}.
\end{equation}
\end{proposition}

Having this proposition at hand it is straightforward to show that \((V, d)\) is not totally bounded.

\begin{corollary}\label{cor:non-cpt}
\((V, d)\) is not totally bounded.
\end{corollary}

\begin{proof}
We define the sets \(F_n\) inductively as follows. We set \(F_1=\{y(1)\}\) and for \(n\geq 2\) we let \(F_n=\{y(n^\ast)\} \cup F_{n-1}\), where \(n^\ast\) is an integer such that
\[
d(x, y(n^\ast))> \tfrac{1}{4}
\]
for all \(x\in F_{n-1}\). The existence of such a point \(y(n^\ast)\) is guaranteed by Propositon~\ref{prop:main}, since \(F_{n-1}\) is a finite set. Let \(F\) be the countable union of the sets \(F_n\). Since \(F\) is infinite and \((1/4)\)-separated, it follows that \((V, d)\) is not totally bounded. 
\end{proof}

In the remainder of this section we prove Proposition~\ref{prop:main}. For each \(n\geq 1\) we set
\[
A_n=\big\{ y(n), \, l(n, k),\, r(n, k) : k\in\{1, \ldots, n\}\big\}.
\]
In particular, \(A_1=\{l, \, r, \, y(1)\}\) and
\[
A_2=\{l(1,1),\, l(2,2),\, y(2),\, r(2,2),\, r(1,1)\}.
\]
Now, let \(x\in V\) and \(y\in V\setminus V_0\) with parents \(c\) and \(d\). If
\[
d(x, y)=\tfrac{1}{2} d(x, c)+\tfrac{1}{2}d(x, d),
\]
then we say that \textit{\(x\) wins}. Moreover, we use the convention that “\(x\) wins” by default if \(y=l\) or \(y=r\).

\begin{definition}
Given \(x\in V\), we let \(n(x)\) denote the smallest positive integer such that for all 
\(y\) contained in \(\bigcup_{n\geq n(x)} A_n\), we have that “\(x\) wins” in the computation of \(d(x,y)\). If no such integer exists, we set \(n(x)=\infty\).
\end{definition} 

Clearly, \(n(l)=n(r)=1\). We will show by induction that \(n(x)\) is finite for every \(x\in V\).

\begin{lemma}\label{lem:main-aux}
Using the definitions from above, \(n(x)\) is finite for all \(x\in V\).
\end{lemma}

\begin{proof}

We say that \(E\subset V\) is \textit{downward closed} if whenever \(x\in E\) and \(a\), \(b\) denote the parents of \(x\), then \(a\), \(b\in E\) as well. In other words, using standard notation from set theory (see e.g. \cite[p. 9]{jech--2003}) and considering the elements of \(V_0\) as ur-elements, we have that \(E\) is downward closed if \(\bigcup E\subset E\). 

Let \(E\subset V\) be a finite subset and \(x\in V\setminus V_0\). Suppose that \(E\cup \{x\}\) is downward closed and 
\[
N=\max\limits_{x'\in E}\, n(x')
\]
is finite. We claim that 
\begin{equation}\label{eq:desired-bound}
n(x)\leq N+1.
\end{equation}
In the following, we use the shorthand notation \(xy=d(x, y)\). Let \(y\in A_n\) for some \(n\geq N+1\) and let \(c\), \(d\) denote the parents of \(y\). We use the convention that \(l\) and \(r\) are their own parents (that is, e.g. \(c=d=l\) if \(y=l\)) and we may suppose that \(lc \leq ld\). Let \(a\), \(b\in E\) denote the parents of \(x\). Thus,  since \(y\), \(c\), \(d\in A_n\cup A_{n-1}\) and \(n-1\geq N \geq  \max\{n(a),\, n(b)\}\), we find that
\begin{align*}
ya+yb&=\tfrac{1}{2}ac+\tfrac{1}{2}ad+\tfrac{1}{2}bc+\tfrac{1}{2}bd \\
&=\tfrac{1}{2}(ac+bd)+\tfrac{1}{2}(ad+bc).
\end{align*}
Let \(c_1\), \(c_2\), and \(d_1\), \(d_2\) denote the parents of \(c\) and \(d\), respectively. We may suppose that \(lc_1\leq lc_2\) and \(ld_1 \leq ld_2\). Now, the crucial observation is that by the recursive definition of \(l(n,k)\), \(r(n, k)\), \(y(n)\), it follows that \(c_2=d_1\). See Figure~\ref{fig:drawing}. Hence, we find that 
\begin{align*}
ac+bd&=\tfrac{1}{2}ac_1+\tfrac{1}{2}ac_2+\tfrac{1}{2}bd_1+\tfrac{1}{2}bd_2 \\
&=\tfrac{1}{2}( ac_1+ bc_2)+\tfrac{1}{2}(ad_1+ bd_2) \\
&\geq xc+xd
\end{align*}
and analogously,
\begin{align*}
ad+bc&=\tfrac{1}{2}ad_1+\tfrac{1}{2}ad_2+\tfrac{1}{2}bc_1+\tfrac{1}{2}bc_2 \\
&=\tfrac{1}{2}( ac_2+bc_1)+\tfrac{1}{2}( ad_2+bd_1)\\
&\geq xc+xd.
\end{align*}
By the above, this implies that \(xy=\frac{1}{2}(xc+xd)\), that is, “\(x\) wins”. Since \(n\geq N+1\) and \(y\in A_n\) were arbitrary, the desired inequality \eqref{eq:desired-bound} follows. 

Now, let \(x\in V\setminus V_0\) and let \(E\) denote the set of all ancestors of \(x\) (including \(x\) itself). Clearly, \(l\), \(r\in E\) and there exists an enumeration \(\{x_1, \ldots, x_m\}\) of \(E\) such that for every \(i\in \{1, \ldots, m\}\),  the set \(E_i=\{x_1, \ldots, x_i\}\) is downward closed. Using induction, it follows from the above that for all \(i\in\{1, \ldots, m\}\), we have that \(n(x')\) is finite for every \(x'\in E_i\). Hence, in particular, as \(x\in E_m\), it follows that \(n(x)\) is finite, as desired. 
\end{proof}

\begin{proof}[Proof of Proposition~\ref{prop:main}]
We set \(n_0=n(x)\). Lemma~\ref{lem:main-aux} tells us that \(n_0\) is a well-defined positive integer. To avoid some technicalities, we may suppose that \(n_0>2\). For \(n\geq n_0\) we set \(m=n_0+n\). By definition of \(n_0\), 
\begin{align*}
d(x, y(m))&=\frac{1}{2}d(x, l(m,m))+\frac{1}{2}d(x, r(m, m))\\
&=\frac{1}4\big( d(x, l(m, m-1))+2\cdot d(x, y(m-1))+d(x, r(m, m-1))\big).
\end{align*}
If \(a\), \(b\) denote the parents of \(y(m)\) and we use the notation \(aa\), \(ab\), and \(ba\), \(bb\) to denote the parents of \(a\) and \(b\), respectively, we find that
\begin{equation}\label{eq:lattice-path}
d(x, y(m))=\frac{1}{2^2}\big( d(x,aa)+d(x,ab)+d(x, ba)+d(x, bb)\big).
\end{equation}
A path \((v_0, \ldots, v_k)\) in \(\Z^2\) is called a \textit{north-east lattice path} if \(v_{i}-v_{i-1}\) is contained in \{(1, 0), (0, 1)\} for each \(i=1, \ldots, k\). As can be seen from \eqref{eq:lattice-path}, the distance \(d(x, y(m))\) can be encoded by considering north-east lattice paths in \(\Z^2\). 

Fix \(n\geq n_0\) and let \(\mathcal{P}\) denote the set of north-east lattice paths in \(\Z^2\) of length \(2n\) starting from \((0,0)\). We define \(f\colon \mathcal{P}\to V\) by setting 
\[
f(\alpha)=z_{2n},
\]
where \(\alpha=(v_0, \ldots, v_{2n})\in \mathcal{P}\) and \(z_0, \ldots, z_{2n}\in V\) are points defined inductively by setting \(z_0=y(m)\) and for \(i\in\{0, \ldots, 2n-1\}\),
\[
z_{i+1}=
\begin{cases}
a_{i} & \text{ if } v_{i+1}-v_{i}=(0,1) \\
b_{i} & \text{ if } v_{i+1}-v_{i}=(1,0),
\end{cases}
\]
 where \(a_{i}\), \(b_{i}\) are the parents of \(z_{i}\) with the convention that \(d(0, a_{i})\leq d(0, b_{i})\) and \(a_{i}=b_{i}=z_{i}\) if \(z_{i}=l\) or \(z_{i}=r\). By the above,
\[
d(x, y(m))=\frac{1}{\abs{\mathcal{P}}} \sum_{\alpha\in \mathcal{P}} d(x, f(\alpha)).
\]
We observe that \(f(\alpha)\) only depends on the endpoint of \(\alpha\). For \(k\in \{-n, \ldots, n\}\), let \(\mathcal{P}_k\subset \mathcal{P}\) denote the set of all paths in \(\mathcal{P}\) that have \((n+k, n-k)\) as endpoint. Clearly,
\[
f(\mathcal{P})=\bigcup_{k=-n}^n f(\mathcal{P}_k).
\]
It is easy to check that \(f(\mathcal{P}_{0})=\{y(n_0)\}\). Moreover, since \(2n \geq m\), it follows that for all \(k\in\{1, \ldots, n_0-1\}\),
\[
f(\mathcal{P}_{k})=\{\,l\big(n_0+k,\, n_0-k+1\big)\},
\]
and for all \(k\in\{n_0, \ldots, n\}\),
\[
f(\mathcal{P}_{k})=\{l\}.
\]
By symmetry, the analogous results also hold for \(k\in \{-1, \ldots, -n\}\) with \(l\) replaced by \(r\). Hence,
\begin{align}
d(x, y(m))&\geq \frac{1}{\abs{\mathcal{P}}}\sum_{k=n_0}^n \,\abs{\mathcal{P}_k}\cdot d(x, l)+\frac{1}{\abs{\mathcal{P}}}\sum_{k=-n}^{-n_0}\, \abs{\mathcal{P}_k} \cdot d(x, r) \nonumber \\
&=\frac{1}{2}\Big(1-\frac{1}{4^n}\sum_{k=-n_0+1}^{n_0-1} \abs{\mathcal{P}_k}\Big), \label{eq:lower-bound}
\end{align}
where we used that \(\abs{\mathcal{P}}=2^{2n}=4^n\) and \(d(x,l)+d(x,r)\geq d(l, r)=1\).
Notice that \(\abs{\mathcal{P}_k}=\binom{2n}{n-k}\), see e.g. \cite[footnote on p. 260]{koshy--2009}, and thus \(\abs{\mathcal{P}_k}\leq \abs{\mathcal{P}_0}=\binom{2n}{n}\). Using Stirling’s formula, we find that \(\binom{2n}{n}\) is asymptotically equal to \(4^n\cdot (n\pi)^{-1/2}\). Thus, it follows from \eqref{eq:lower-bound} that
\[
\limsup_{m\to \infty} \, d(x, y(m)) \geq \frac{1}{2},
\]
as desired. 
\end{proof}

\section{Conical midpoint maps}\label{sec:conical-midpoint} 
In this section we introduce conical midpoint maps and derive some of their basic properties. Let \(X=(X, d)\) be a metric space. We use \(\olsi{X}\) to denote the metric completion of \(X\). If readability demands it we will sometimes tacitly identify \(X\) with its canonical isometric copy in \(\olsi{X}\). Any conical midpoint map on a metric space \(X\) induces a conical bicombing on \(\olsi{X}\). This is discussed at the end of this section. 

\begin{definition}\label{def:conical-midpoint-map}
We say that \(m\colon X\times X \to X\) is a \textit{conical midpoint map} if for all \(x\), \(y\), \(z\in X\), the following holds:
\begin{enumerate}
\item\label{it:ax1} \(m(x,x)=x\),
\item\label{it:ax2} \(m(x,y)=m(y,x)\),
\item\label{it:ax3} \(d(m(x,y), m(x, z)) \leq \frac{1}{2} d(y,z)\).
\end{enumerate}
\end{definition}

We remark that for midpoints in Euclidean space, the inequality in \eqref{it:ax3} becomes an equality. 
It is easy to see that if \(m\) is as in Definition~\ref{def:conical-midpoint-map}, then \(z=m(x_1, x_2)\) is a midpoint of \(x_1\) and \(x_2\). Furthermore, \eqref{it:ax3} can be upgraded to a more general inequality involving four points. Indeed, for all \(x\), \(y\), \(x'\), \(y'\in X\), it holds
\begin{equation}\label{eq:improved-inequality}
d(m(x,y), m(x', y')) \leq \tfrac{1}{2} d(x,x')+\tfrac{1}{2} d(y,y').
\end{equation}
In what follows, we show that conical midpoint maps induce conical bicombings in a natural way. The used recursive construction is well-known and goes back to Menger (see \cite[Section 6]{menger-1928}). 

Let \(m\) be a concial midpoint map on \(X\) and \(x\), \(y\in X\).  Further, let \(\mathcal{G}_n= (2^{-n}\cdot \Z) \cap [0,1]\), where \(n\geq 0\), be the \(2^{-n}\)-grid in \([0,1]\). We define \(\sigma_{xy}\colon \bigcup \mathcal{G}_n\to X\) recursively as follows. We set \(\sigma_{xy}(0)=x\), \(\sigma_{xy}(1)=y\) and if \(t\in \mathcal{G}_{n}\setminus \mathcal{G}_{n-1}\), then we set
\[
\sigma_{xy}(t)=m(\sigma_{xy}(r), \sigma_{xy}(s)),
\]
where \(r\), \(s\in \mathcal{G}_{n-1}\) are the unique points such that \(t=\frac{1}{2}r+\frac{1}{2}s\) and \(\abs{r-s}=2^{-(n-1)}\). 

\begin{lemma}\label{lem:partial-step-construction-of-induced-bicombing}
The map \(\sigma_{xy}\) extends uniquely to a geodesic \(\olsi{\sigma}_{xy}\colon [0,1]\to \olsi{X}\). Moreover,
\begin{equation}\label{eq:conical-2}
d(\olsi{\sigma}_{xy}(t), \olsi{\sigma}_{x'y'}(t)) \leq (1-t) d(x, x')+t\, d(y, y')
\end{equation}
for all \(x\), \(y\), \(x'\), \(y'\in X\) and all \(t\in [0,1]\).
\end{lemma}
\begin{proof}
To begin, we show by induction that \(\sigma_{xy}|_{\mathcal{G}_n}\) is an isometric embedding. This is clearly true for \(n=0\). Now, fix \(t_i\in \mathcal{G}_{n}\), \(i=1,2\) and let \(r_i\), \(s_i\in \mathcal{G}_{n-1}\) with \(s_i \leq r_i\) be such that \(t_i=\frac{1}{2}s_i+\frac{1}{2}r_i\) and \(\sigma_{xy}(t_i)=m(\sigma_{xy}(s_i), \sigma_{xy}(r_i))\). Without loss of generality, we may suppose that \(t_1 \leq t_2\).  Using the triangle inequality, we get
\begin{align*}
d(\sigma_{xy}(t_1), \sigma_{xy}(t_2)) &\leq d(\sigma_{xy}(t_1), \sigma_{xy}(r_1))+d(\sigma_{xy}(r_1), \sigma_{xy}(s_2))\\
&\phantom{\leq}+d(\sigma_{xy}(s_2), \sigma_{xy}(t_2)),
\end{align*}
and so, by the induction hypothesis and because \(m\) is a midpoint map,
\begin{align*}
d(\sigma_{xy}(t_1), \sigma_{xy}(t_2)) &\leq \big(\frac{r_1-s_1}{2}+\abs{s_2-r_1}+\frac{r_2-s_2}{2}\big)d(x,y).
\end{align*}
But, since \(t_1 \leq t_2\), we have \(r_1 \leq s_2\). Hence, by the above, \(d(\sigma_{xy}(t_1), \sigma_{xy}(t_2)) \leq \abs{t_1-t_2} d(x,y)\). As a result,
\begin{align*}
d(x, y) &\leq d(x, \sigma_{xy}(t_1))+d(\sigma_{xy}(t_1), \sigma_{xy}(t_2))+d(\sigma_{xy}(t_2), y) \\
&\leq \big( t_1+\abs{t_1-t_2}+\abs{t_2-1}\big)d(x, y).
\end{align*}
This implies that \(d(\sigma_{xy}(t_1), \sigma_{xy}(t_2))=\abs{t_1-t_2} d(x,y)\), and so \(\sigma_{xy}|_{\mathcal{G}_{n}}\) is an isometric embedding, as claimed. Hence, \(\sigma_{xy}\) can be uniquely extended to an isometric embedding \(\olsi{\sigma}_{xy}\colon[0,1]\to \olsi{X}\). Next, we show \eqref{eq:conical-2}. Clearly,
\[
2\cdot d(\olsi{\sigma}_{xy}(1/2), \olsi{\sigma}_{x' y'}(1/2)) \leq  d(x, x')+ d(y, y'),
\]
as \(\olsi{\sigma}_{xy}(1/2)=m(x, y)\), \(\olsi{\sigma}_{x'y'}(1/2)=m(x', y')\) and  \(m\) is conical midpoint map and thus satisfies \eqref{eq:improved-inequality}. We now proceed by induction and show that if \eqref{eq:conical-2} is valid for all \(t\in \mathcal{G}_{n-1}\), then it is also valid for all \(t\in \mathcal{G}_{n}\). Fix \(t\in \mathcal{G}_n\) and let \(r\), \(s\in \mathcal{G}_{n-1}\) be the unique points with \(s \leq r\) such that \(t=\frac{1}{2}s+\frac{1}{2}r\). We compute
\begin{align*}
2\cdot d(\olsi{\sigma}_{xy}(t), \olsi{\sigma}_{x' y'}(t)) &\leq d(\olsi{\sigma}_{xy}(s), \olsi{\sigma}_{x' y'}(s))+ d(\olsi{\sigma}_{xy}(r), \olsi{\sigma}_{x' y'}(r)) \\
& \leq ( (1-s)+ (1-r) )d(x, x')+( s+r) d(y, y');
\end{align*}
hence, \eqref{eq:conical-2} holds for all \(t\in \mathcal{G}_n\). Since \(\bigcup \mathcal{G}_n\) is a dense subset of \([0,1]\) and \(\olsi{\sigma}_{xy}\) and \(\olsi{\sigma}_{x' y'}\) are geodesics, \eqref{eq:conical-2} is valid for all \(t\in [0,1]\). 
\end{proof}

Thus, we have constructed a map \(\olsi{\sigma} \colon X\times X \times [0,1]\to \olsi{X}\) such that \eqref{eq:conical} holds for all geodesics \(\olsi{\sigma}_{xy}\) and \(\olsi{\sigma}_{x'y'}\). Now, given \(x\), \(y\in \olsi{X}\), we set
\[
\olsi{\sigma}_{xy}(t)=\lim_{n\to \infty} \olsi{\sigma}_{x_n y_n}(t)
\]
where \(x_n\), \(y_n\in X\) are points such that \(x_n \to x\) and \(y_n \to y\) as \(n\to \infty\), respectively. It follows that \(\olsi{\sigma}\) is a well-defined reversible conical bicombing on \(\olsi{X}\). We call \(\olsi{\sigma}\) the \textit{bicombing induced by \(m\)}. We point out that \(m\) is defined on an arbitrary metric space \(X\) but \(\olsi{\sigma}\) is always a bicombing on \(\olsi{X}\).

We conclude this section by giving a description of \(\sigma\)-convex hulls in terms of \(m\). For any \(A\subset X\), we set \(\mathcal{M}_1(A)=\big\{ m(a, a') : a, a'\in A\big\}\) and 
\[
\mathcal{M}_n(A)=\mathcal{M}_1(\mathcal{M}_{n-1}(A))
\]
for all \(n\geq 2\). We let \(\convm(A) \subset \olsi{X}\) denote the closure of the countable union of the sets \(\mathcal{M}_n(A)\). We emphasize that the closure is taken in \(\olsi{X}\) and not in \(X\). The following result shows that this construction recovers the closed \(\sigma\)-convex hull of \(A\).

\begin{lemma}\label{lem:induced-bicombing}
Let \(m\) be a conical midpoint map on a metric space \(X\) and suppose \(\sigma\) denotes the bicombing on \(\olsi{X}\) induced by \(m\). Then
\[
\conv(A)=\convm(A)
\]
for all \(A\subset X\). 
\end{lemma}
\begin{proof}
Clearly, \(\convm (A)\subset \conv(A)\). Thus, it suffices to show that the closed set \(\convm (A)\) is \(\sigma\)-convex. To this end, let \(n\geq 1\) and let \(x\), \(y\in \mathcal{M}_n(A)\). Let the \(2^{-s}\)-grid \(\mathcal{G}_s\), for \(s\geq 1\), be defined as above. By construction of \(\sigma\), it follows that \(\sigma_{xy}(\mathcal{G}_s)\subset \mathcal{M}_{n+s}(A)\). Hence, \(\sigma_{xy}([0,1])\subset \convm(A)\). Now, suppose that \(x\), \(y\in \convm(A)\). There exist points \(x_k\), \(y_k\in \mathcal{M}_{n_k}(A)\) such that \(x_k \to x\) and \(y_k \to y\) as \(k\to \infty\), respectively. Moreover, \(\sigma_{x_k y_k}\to \sigma_{xy}\) uniformly. This implies that \(\sigma_{xy}([0,1])\subset \convm(A)\), and so \(\convm(A)\) is \(\sigma\)-convex.    
\end{proof}

\section{Proof of main results}\label{sec:proof-of-main-theorems}

In this section we prove the main results from the introduction. Theorem~\ref{thm:non-cpmpactness} is an immediate consequence of the following result.

\begin{theorem}\label{thm:Z-non-cpt}
Let \(j\) be a positive integer and let \(X\) be the completion of the metric space \((V, d)\) constructed in Section~\ref{sec:appending-midpoints}. Then \(X\) admits a conical bicombing \(\sigma\) and there is a finite subset \(A\subset X\) consisting of \(j\) points such that \(\conv(A)=X\). Moreover, \(X\) is non-compact for every \(j>1\).
\end{theorem}

\begin{proof}
Let \((V, d)\) be the metric space constructed in Section~\ref{sec:appending-midpoints}. By definition of \(d\), it follows that \(m\colon V\times V\to V\) defined by \((a,b)\mapsto m(a,b)\) is a conical midpoint map. Thus, it follows from the results of Section~\ref{sec:conical-midpoint} that \(X=\olsi{V}\) admits a conical bicombing \(\sigma\). We set \(A=V_0\). Clearly, \(A\) consists of \(j\) points and 
\[
\bigcup_{n\geq 1} \mathcal{M}_n(A)=V.
\]
Hence, it follows from Lemma~\ref{lem:induced-bicombing} that \(\conv(A)=X\). Moreover, by Corollary~\ref{cor:non-cpt} we know that \((V, d)\) is not totally bounded if \(j> 1\). Hence, \(X=\conv(A)\) is non-compact for every \(j>1\), as desired. 
\end{proof}

A metric space \(Y\) is called \textit{injective} if whenever \(A\subset X\) are metric spaces and \(f\colon A\to Y\) a \(1\)-Lipschitz map, then there exists a \(1\)-Lipschitz extension \(F\colon X\to Y\) of \(f\). Injective metric spaces have been introduced by Aronszajn and Panitchpakdi in \cite{aronszajn-1956} and are sometimes also called hyperconvex metric spaces by some authors. We refer to \cite{espinola-2001, lang-2013} for basic properties of these spaces. As observed by Lang in \cite[Proposition~3.8]{lang-2013}, every injective metric spaces admits a conical bicombing. Indeed, given an injective metric space \(Y\), by applying Kuratowski's embedding theorem, we may suppose that \(Y\subset C_b(Y)\), and so because \(Y\) is injective, there is a \(1\)-Lipschitz retraction \(r\colon C_b(Y) \to Y\) and thus
\[
\sigma(x, y, t)=r((1-t)x+t y)
\]
defines a conical bicombing on \(Y\). Using an extension result of \cite{basso2020extending}, we find that Theorem~\ref{thm:non-cpmpactness} is also valid for an injective metric space.

\begin{theorem}\label{thm:extension-to-injective-metric-spaces}
There exists an injective metric space \(Y\) that admits a reversible conical bicombing \(\sigma\) such that \(Y\) has a finite subset whose closed \(\sigma\)-convex hull is non-compact. 
\end{theorem}

\begin{proof}
Let \(j> 1\) and let \(X\) be the completion of the metric space \((V, d)\) constructed in Section~\ref{sec:appending-midpoints}. We recall that \(V\) is naturally equipped with a conical midpoint map \(m\). Let \(\sigma\) denote the conical bicombing on \(X\) induced by \(m\). As \(m\) is symmetric, it follows that \(\sigma_{xy}(t)=\sigma_{yx}(1-t)\) for all \(x\), \(y\in X\). This shows that \(\sigma\) is a reversible conical bicombing. Hence, by virtue of \cite[Theorem~1.2]{basso2020extending}, there exists an injective metric space \(Y\) containing \(X\), and a conical bicombing \(\widetilde{\sigma}\) on \(Y\) such that \(\widetilde{\sigma}_{xy}=\sigma_{xy}\) for all \(x\), \(y\in X\). By looking at the proof of \cite[Theorem~1.2]{basso2020extending}, it is not difficult to see that \(\widetilde{\sigma}\) is reversible. As \(X\) is complete, it follows that \(\convt(A)=\conv(A)\) for any \(A\subset X\). Therefore, due to Theorem~\ref{thm:Z-non-cpt}, \(Y\) admits a finite subset whose closed \(\widetilde{\sigma}\)-convex hull is non-compact.  
\end{proof}

We finish this section by proving the following more general version of Theorem~\ref{thm:universal}.

\begin{theorem}\label{thm:universal-with-general-targets}
Let \(j\) be a positive integer and let \(X\) be the completion of the metric space \((V, d)\) constructed in Section~\ref{sec:appending-midpoints}. Then whenever \(Y\) is a complete metric space with a conical midpoint map \(m_Y\) and \(A\subset Y\) contains at most \(j\) points, then there exists a Lipschitz map \(\Phi\colon X\to Y\) such that \(A\subset \Phi(X)\) and furthermore \(\Phi(X)\) is \(\sigma_Y\)-convex with respect to the conical bicombing \(\sigma_Y\) induced by \(m_Y\).
\end{theorem}

\begin{proof}
We define the map \(\Phi\colon V \to Y\) as follows. We let \(\Phi|_{V_0}\) be any surjection onto \(A\), and we define \(\Phi\) recursively by setting \(\Phi(x)=m_Y(\Phi(a), \Phi(b))\) for \(x\in V\setminus V_0\) with parents \(a\), \(b\). In the following, we show by induction that \(\Phi\) is \(L\)-Lipschitz for \(L=\diam(A)\). 

Recall that we use \(i(x)\) to denote the index of \(x\in V\). Let \(\phi \colon V\times V \to \R\) be defined by \(\phi(x,y)=i(x)+i(y)\) and let \(R_n\) denote the union of \(\phi^{-1}(k)\) for \(k=0, \ldots, n\). Now, let \((x,y)\in V\times V\) and consider the inequality
\begin{equation}\label{eq:inductive-bound-Rn}
d(\Phi(x), \Phi(y))\leq L\cdot d(x, y).
\end{equation}
Clearly, this holds for all \((x,y)\in \phi^{-1}(0)\). We show in the following that \eqref{eq:inductive-bound-Rn} is valid for all \((x, y)\in R_n\) provided it holds for all points in \(R_{n-1}\). Let \((x, y)\in \phi^{-1}(n)\). By construction of \(d\), there exist \((u, v)\), \((p, q)\in R_{n-1}\) such \(m(u,v)=x\) and \(m(p, q)=y\), as well as
\[
d(x, y)=\frac{1}{2} d(u, p)+\frac{1}{2} d(v, q).
\]
Thus, by construction of \(\Phi\), we have that \(m_Y(\Phi(u), \Phi(v))=\Phi(x)\) and \(m_Y(\Phi(p), \Phi(q))=\Phi(y)\), and thus using that \(m_Y\) is a conical midpoint map,
\[
d(\Phi(x), \Phi(y))\leq \frac{1}{2} d(\Phi(u), \Phi(p))+\frac{1}{2} d(\Phi(v), \Phi(q)).
\]
Hence, because we assume that \eqref{eq:inductive-bound-Rn} is valid for all points in \(R_{n-1}\), we get
\[
d(\Phi(x), \Phi(y)) \leq \frac{L}{2} d(u, p)+\frac{L}{2} d(v, q)=L\cdot d(x,y).
\]
Since \(R_n=\phi^{-1}(n)\cup R_{n-1}\), we find that \eqref{eq:inductive-bound-Rn} is valid for all \((x, y)\in R_n\). It thus follows by induction that \(\Phi\) is \(L\)-Lipschitz. Since \(Y\) is a complete metric space, \(\Phi\) can be uniquely extended to an \(L\)-Lipschitz map \(X \to Y\) which for simplicity we also denote by \(\Phi\).

To finish the proof we show that \(\Phi(X)\) is \(\sigma_Y\)-convex. By construction of \(\Phi\) and since both \(\sigma_X\) and \(\sigma_Y\) are induced by a conical midpoint map, it follows that \(\Phi(\sigma_X(x, y, t))=\sigma_Y(\Phi(x), \Phi(y), t)\) for all \(x\), \(y\in V\) and all \(t\in [0,1]\). Let now \(x\), \(y\in X\) be arbitrary. Then there exists \(x_k\), \(y_k\in V\) such that \(x_k \to x\) and \(y_k \to y\) as \(k\to \infty\), respectively. Moreover, \(\sigma_{x_k y_k}\to \sigma_{xy}\) uniformly. Hence, as \(\Phi\) is Lipschitz continuous, we have \(\Phi(\sigma_X(x, y, t))=\sigma_Y(\Phi(x), \Phi(y), t)\) for all \(t\in [0,1]\). This shows that \(\Phi(X)\) is \(\sigma_Y\)-convex. 
\end{proof}

\section{On a question of Descombes and Lang}\label{sec:bicombings}

In practice, it is often desirable to impose stronger properties on a bicombing than \eqref{eq:conical}. See \cite{chalopin2020helly, huang2020morse, kleiner-2020, haettel2022link} for some recent examples. Descombes and Lang \cite{descombes-lang-2015} introduced the following notions:

\begin{enumerate}
    \item  if \eqref{eq:conical} holds, then \(\sigma\) is said to be \textit{conical}. \vspace{0.25em}
    \item if for all \(x\), \(y\), \(x'\), \(y'\in X\), the map \(t\mapsto d(\sigma_{xy}(t), \sigma_{x'y'}(t))\) is convex on \([0,1]\),  then \(\sigma\) is called \textit{convex}. \vspace{0.25em}
    \item if
    \[
    \sigma(\sigma(x, y, r), \sigma(x, y, s), t)=\sigma(x, y, (1-t)r+ts)
    \]
    for all \(x\), \(y\in X\) and all \(r\), \(s\), \(t\in [0,1]\) with \(r<s\), then \(\sigma\) is called \textit{consistent}. \vspace{0.15em}
\end{enumerate}

Consistent bicombings are used in \cite{foertsch-2008, lytchak-2007},
and a variant of the definition that allows for a bounded error is studied in \cite[Definition 2.6]{engel2021coronas}. We do not know if every space with a bicombing also admits a consistent bicombing\footnote{This innocent looking question seems to be quite difficult on closer inspection}. Clearly, the following implications hold
\begin{equation}\label{eq:implications}
\sigma \text{ consistent \& convex} \Longrightarrow \sigma \text{ convex} \Longrightarrow \sigma \text{ conical}.
\end{equation}
As it turns out, both reverse implications do not hold. Indeed, there are many examples of conical bicombings that are not convex (see \cite[Example 2.2]{descombes-lang-2015} and \cite[Example 3.6]{basso2020extending}). Hence, the second reverse implication in \eqref{eq:implications} does not hold. Moreover, as is demonstrated in \cite[Theorem~1.1]{basso-miesch-2019} there exist reversible convex bicombings that are not consistent. Thus, the first reverse implication in \eqref{eq:implications} also does not hold. 

However, the authors are not aware of any example of a metric space with a conical bicombing that does not also admit a consistent conical bicombing. In other words, the following question of Descombes and Lang \cite{descombes-lang-2015} is still open.

\begin{question}[Descombes--Lang]\label{qe:descombes-lang}
Let \(X\) be a complete metric space. Is it true that \(X\) admits a conical bicombing if and only if it admits a consistent conical bicombing?
\end{question}

This question also appears in the problem list \cite[p. 385]{oberwolfach-report-2021}. A partial result that indicates an affirmative answer when \(X\) is proper has been obtained in \cite[Theorem~1.4]{basso2020extending}. One difficulty in finding a negative answer to Question~\ref{qe:descombes-lang} lies in the fact that many known examples of metric spaces with a conical bicombing have locally a nice structure. In this situation one can then employ a generalized version of the Cartan-Hadamard theorem \cite{miesch-2017} to construct a consistent conical bicombing. The metric space \((V, d)\) is locally not 'nice' as it is fractal-like in nature. So we believe that its completion could be a potential candidate for a counterexample to Question~\ref{qe:descombes-lang}.

\let\oldbibliography\thebibliography
\renewcommand{\thebibliography}[1]{\oldbibliography{#1}
\setlength{\itemsep}{4pt}} 
\bibliographystyle{plain}
\bibliography{sample}

\vspace{1em}

{\small\noindent G. Basso (\texttt{giuliano.basso@web.de}) \\
 Max Planck Institute for Mathematics,
Vivatsgasse 7,
53111 Bonn,
Germany\\}

{\small\noindent
Y. Krifka (\texttt{krifka@mpim-bonn.mpg.de}) \\
 Max Planck Institute for Mathematics,
Vivatsgasse 7,
53111 Bonn,
Germany\\}

{\small\noindent E. Soultanis (\texttt{elefterios.soultanis@gmail.com}) \\
Department of Mathematics and Statistics, University of Jyväskylä, P.O. Box 35, FI-40014 University of Jyväskylä, Finland}

\end{document}